\documentclass[a4paper,10pt]{article}
\usepackage[utf8]{inputenc}

\providecommand{\keywords}[1]{\textbf{\textit{keywords:}} #1}

\usepackage{float}

\usepackage{amsmath,amsfonts,amssymb,amsthm,epsfig,epstopdf,titling,url,array,tabularx}
\usepackage{caption,subcaption,sublabel}
\usepackage{multirow,multicol}
\usepackage{graphicx}
\usepackage{pgfplots}
\usepackage{algorithmic} 
\usepackage[boxruled,longend,linesnumbered,ruled]{algorithm2e}

\theoremstyle{definition}

\newtheorem{theorem}{Theorem}[section]

\newtheorem{lemma}{Lemma}[section]

\theoremstyle{remark}

\DeclareMathOperator{\Rel}{Re}
\DeclareMathOperator{\Imag}{Im}

\usepackage{authblk}

\title{Computationally efficient optimal control for unstable power system models}

\author[,1,3]{Mahtab Uddin\thanks{Corresponding author, mahtab@ins.uiu.ac.bd}}
\author[2]{M. Monir Uddin}
\author[3]{Md. Abdul Hakim Khan}

\affil[1]{Institute of Natural Sciences, United International University, Dhaka-1212, Bangladesh}
\affil[2]{Department of Mathematics and Physics, North south University, Dhaka-1229, Bangladesh}
\affil[3]{Department of Mathematics, Bangladesh University of Engineering \& Technology, Dhaka-1000, Bangladesh}

\date{}

\begin{document}
	
\maketitle             

\begin{abstract}
	
In this article, the focus is mainly on gaining the optimal control for the unstable power system models and stabilizing them through the Riccati-based feedback stabilization process with sparsity-preserving techniques. We are to find the solution of the Continuous-time Algebraic Riccati Equations (CAREs) governed from the unstable power system models derived from the Brazilian Inter-Connected Power System (BIPS) models, which are large-scale sparse index-1 descriptor systems. We propose the projection-based Rational Krylov Subspace Method (RKSM) for the iterative computation of the solution of the CAREs. The novelties of RKSM are sparsity-preserving computations and the implementation of time-convenient adaptive shift parameters. We modify the Low-Rank Cholesky-Factor integrated Alternating Direction Implicit (LRCF-ADI) technique based nested iterative Kleinman-Newton (KN) method to a sparse form and adjust this to solve the desired CAREs. We compare the results achieved by the Kleinman-Newton method with that of using the RKSM. The applicability and adaptability of the proposed techniques are justified numerically with MATLAB simulations. Transient behaviors of the target models are investigated for comparative analysis through the tabular and graphical approaches. 	    
	
\end{abstract}

\begin{center}
\end{center}

\section{Introduction}

For the practical purposes, optimal control is a vital part of the engineering interest e.g., industrial control systems, system defense strategy, voltage stability of the power systems, and signal processing \cite{huang2002new,angle2019identifying}. Multi-tasking systems having various components arise in many fields of engineering applications, such as microelectronics, micro-electro-mechanical systems, cybersecurity, computer control of industrial processes, communication systems, etc. These systems are composed of branches of sub-systems and are functioned by very large mathematical models utilizing the interrelated inner mathematical system of higher dimensions.

Power system models are one of the prime branches of the application of optimal controls. For the multi-connected power systems, avoiding cyber-attacks, ensuring the stability of the power connections, and maintaining the compatible frequency level are essential. For those inevitable issues, optimal controls have exigent roles \cite{cui2020deeply}.

The dynamic of a large-scale power system model can be described by the Differential Algebraic Equations (DAEs) as
\begin{equation} \label{eqn:powersystem}
\begin{aligned}
\dot{x}(t) & = f(x_1,x_2,P), \quad f:\mathbb{R}^{n_1 + n_2 + n_3}\rightarrow \mathbb{R}^{n_1},\\
0 & = g(x_1,x_2,P), \quad g:\mathbb{R}^{n_1 + n_2 + n_3}\rightarrow \mathbb{R}^{n_2},
\end{aligned}
\end{equation}
where $x_1\in X_1 \subset \mathbb{R}^{n_1}$ is the vector with differential variables, $x_2\in X_2 \subset \mathbb{R}^{n_2}$ is the vector with algebraic variables and $P \in \mathbb{R}^{n_3}$ is the vector of parameters with $n_1 + n_2 =n$ \cite{kose2003introduction,gross2018steady}. In the state-space representation, the dynamic state variables $x_1$ and instantaneous variables $x_2$ are defined for the specific system, where the parameter $P$ defines the configuration and the operation condition.  

The state variables in $x_1$ are time-dependent generator voltages and the parameter $P$ is composed of the system parameters. The control devices together form $ f(x_1,x_2,P)$ and the power flow balance form $g(x_1,x_2,P)$. In case of voltage stability, some equations of the $f(x_1,x_2,P)$ will not be considered. For a fixed parameter $P_0$, linearizing the system (\ref{eqn:powersystem}) around the equilibrium point will provide the following Linear Time-Invariant (LTI) continuous-time system with the input-output equations in the sparse form with the block matrices as  
\begin{equation}  \label{eqn:DS_matrix_vector}
\begin{aligned}
\underbrace{ \begin{bmatrix}  E_1  &  0 \\ 0  &  0   \end{bmatrix}}_E
\underbrace{ \begin{bmatrix}  \dot{x_1}(t)  \\ \dot{x_2}(t)    \end{bmatrix}}_{\dot{x}(t)}
&= \underbrace{\begin{bmatrix}  J_1 &  J_2 \\ J_3 & J_4 \end{bmatrix}}_A
\underbrace{ \begin{bmatrix}    x_1(t) \\ x_2(t)  \end{bmatrix}}_{x(t)}
+ \underbrace{\begin{bmatrix}    B_1 \\ B_2  \end{bmatrix}}_B u(t); x(t_0)=x_0,\\
y(t)&= \underbrace{\begin{bmatrix}  C_1  & C_2 \end{bmatrix}}_C
\begin{bmatrix}    x_1(t) \\ x_2(t)  \end{bmatrix} + D u(t),
\end{aligned}
\end{equation} 
where $E,A\in \mathbb{R}^{n\times n}$, $B\in \mathbb{R}^{n\times p}$, $C\in \mathbb{R}^{m\times n}$ and $D\in \mathbb{R}^{m\times p}$ with very large $n$ and $p,m\ll n$, represent differential coefficient matrix, state matrix, control multiplier matrix, state multiplier matrix and direct transmission map respectively \cite{mahmud2012full}. In the system (\ref{eqn:DS_matrix_vector}), $x(t)\in \mathbb{R}^n$ is the state vector and $u(t)\in \mathbb{R}^p$ is control (input), while $y(t)\in \mathbb{R}^m$ is the output vector and considering $x(t_0)=x_0$ as the initial state. In most of the state-space representations the direct transmission remains absent and because of that $D=0$. Since $E$ is singular (i.e. $\det(E)=0$), the system (\ref{eqn:DS_matrix_vector}) is called the descriptor system \cite{hossain2019iterative,benner2016reduced}. 

Here $x_1\in \mathbb{R}^{n_1}$, $x_2\in \mathbb{R}^{n_2}$ with $n_1+n_2 = n$ are state vectors and other sub-matrices are sparse in appropriate dimensions. If $E_1$ and $J_1$ have full rank, and $J_4$ is non-singular (i.e. $\det(A)\neq 0$), the system is called the index-1 descriptor system \cite{uddin2015computational}. In the current work, we will focus on the stabilization of index-1 descriptor system only. By proper substitution and elimination the descriptor system (\ref{eqn:DS_matrix_vector})  can be converted to the generalized LTI continuous-time system 
\begin{equation} \label{eqn:d-state-space}
\begin{aligned}
\mathcal{E}\dot{x}(t)&=\mathcal{A}x(t)+\mathcal{B}u(t),\\
y(t)&=\mathcal{C}x(t)+\mathcal{D}u(t),
\end{aligned}
\end{equation}
where, we have considered the following relations 
\begin{equation} \label{eqn:schur_complements}
\begin{aligned}
x:&=x_1, \quad \mathcal{E}:= E_1,\\
\mathcal{A}:&= J_1-J_2{J_4}^{-1}J_3, \quad  
\mathcal{B}:= B_1-J_2{J_4}^{-1}B_2,\\
\mathcal{C}:&= C_1-C_2{J_4}^{-1}J_3, \quad
\mathcal{D}:= D-C_2{J_4}^{-1}B_2.
\end{aligned}
\end{equation}

\begin{lemma} [\textbf{Equivalence of transfer functions \cite{gugercin2013model}}] \label{transfer_function}
Assume the transfer functions $G(s)=C(s E - A)^{-1} B + D$  and  $\mathcal{G}(s)=\mathcal{C}(s \mathcal{E} - \mathcal{A})^{-1} \mathcal{B} + \mathcal{D}$ are obtained from the index-1 descriptor system (\ref{eqn:DS_matrix_vector}) and the converted generalized system (\ref{eqn:d-state-space}), respectively. Then, the transfer functions $G(s)$ and $\mathcal{G}(s)$ are identical and hence those systems are equivalent.
\end{lemma}

Though the systems (\ref{eqn:DS_matrix_vector}) and (\ref{eqn:d-state-space}) are equivalent, the system (\ref{eqn:DS_matrix_vector}) is sparse and the system (\ref{eqn:d-state-space}) is dense. This explicit conversion is contradictory with the aim of the work and it will be bypassed by a more efficient way. 

LTI continuous-time systems are the pivot ingredient of the present control theory and many other areas of science and engineering \cite{benner2016structure}. Continuous-time Algebraic Riccati Equation (CARE) appears in many branches of engineering applications; especially in electrical fields \cite{chu2011solving,chen2016linear}. The CARE connected to the system (\ref{eqn:d-state-space}) is defined as
\begin{equation} \label{eqn:D-GCARE}
\begin{aligned}
\mathcal{A}^T X\mathcal{E}+\mathcal{E}^T X\mathcal{A}-\mathcal{E}^T X\mathcal{B}\mathcal{B}^T X\mathcal{E}+\mathcal{C}^T \mathcal{C}=0.
\end{aligned}
\end{equation} 

If the Hamiltonian matrix corresponding to the system (\ref{eqn:d-state-space}) has no pure imaginary eigenvalues, then the solution $X$ of the CARE (\ref{eqn:D-GCARE}) exists and unique \cite{abou2012matrix}. The solution $X$ of (\ref{eqn:D-GCARE}) is symmetric positive-definite and called stabilizing for the stable closed-loop matrix $\mathcal{A}- (\mathcal{B}\mathcal{B}^T)X\mathcal{E}$. Riccati-based feedback matrix has a prime role in the stabilization approaches for unstable systems \cite{uddin2019riccati,uddin2019efficient}. To find an optimal feedback matrix $K^o$, the Linear Quadratic Regulator (LQR) problem technique can be applied, where the cost functional is defined as
\begin{equation} \label{eqn:LQR}
\begin{aligned}
J(u,x_0)=\int_{0}^{\infty } (x^T(t) \mathcal{C}^T \mathcal{C} x(t)+u^T(t)u(t))dt.
\end{aligned}
\end{equation}

The cost functional (\ref{eqn:LQR}) can be optimized as $J(u^o,x_0)= {x_0}^T Xx_0$ by applying the optimal control $u^o = -K^o x(t)$ generated by the optimal feedback matrix $K^o = \mathcal{B}^T X\mathcal{E}$ associated with the solution matrix $X$ of the CARE (\ref{eqn:D-GCARE}). Using the optimal feedback matrix $K^o$, an unstable LTI continuous-time system can be optimally stabilized by replacing $\mathcal{A}$ by $\mathcal{A}_s = \mathcal{A}-\mathcal{B}K^o$. The stabilized system can be written as
\begin{equation} \label{eqn:st-state-space}
\begin{aligned}
\mathcal{E}\dot{x}(t) = \mathcal{A}_s x(t)+\mathcal{B}u(t), \\
y(t) = \mathcal{C}x(t)+\mathcal{D}u(t).
\end{aligned}
\end{equation}

The eligibility of Rational Krylov Subspace Method (RKSM) for the large-scale LTI continuous-time systems has discussed by Simoncini in \cite{simoncini2007new} and the application of adaptive RKSM to solve large-scale CARE for finding optimal control of the LTI systems has narrated by Druskin \textit{et al.} in \cite{druskin2011adaptive}. Analysis of the basic properties of RKSM for solving large-scale CAREs subject to LTI systems investigated by Simoncini in \cite{simoncini2016analysis}, where the author briefed a new concept of shift parameters that is efficient for the perturbed systems. Very detailed discussion on the numerical solution of large-scale CAREs and LQR based optimal control problems are given by Benner \textit{et al.} in \cite{benner2008numerical}, where the authors narrated the extensions Newton method by means of Alternating Direction Implicit (ADI) technique with convergence properties and the comparative analysis with the RKSM approach. For solving large-scale matrix equations, the Kleinman-Newton method based on Low-Rank Cholesky-Factor ADI (LRCF-ADI) technique is discussed by Kuerschner in \cite{kurschner2016efficient}. 

Since there is a scarcity of efficient computational solvers or feasible simulation tools for large-scale CAREs governed from the unstable power system models, the concentration of this work is to develop sparsity-preserving efficient techniques to find the solution of those CAREs. We are proposing a sparsity-preserving and rapid convergent form of the RKSM algorithm for finding solutions of CAREs associated with the unstable power system models and implementation of Riccati-based feedback stabilization. Also, a modified sparse form of LRCF-ADI based Kleinman-Newton method for solving CAREs subject to unstable power system models and corresponding stabilization approach is proposed. The proposed techniques are applied for the stabilization of the transient behaviors of unstable Brazilian Inter-Connected Power System (BIPS) models \cite{varricchio2015hybrid}. Moreover, the comparison of the computational results will be provided in both tabular and graphical methods.

\section{Preliminaries}\label{pre}

In this section, we discuss the background of the proposed techniques, which are derived for generalized LTI continuous-time systems. It includes the derivation of the Rational Krylov Subspace Method (RKSM) technique for solving the Riccati equation, basic structure of the Kleinman-Newton method, and derivation of Low-Rank Cholesky-Factor Alternative Direction Implicit (LRCF-ADI) method for solving Lyapunov equations. 

\subsection{Generalized RKSM technique for solving Riccati equations}

In \cite{simoncini2016analysis} Simoncini applied RKSM approach for solving the CARE in the form 
\begin{equation} \label{eqn:GCARE}
\begin{aligned}
A^T XE + E^T XA - E^T XBB^T XE + C^T C = 0,
\end{aligned}
\end{equation}
associated with the LTI continuous-time system 
\begin{equation} \label{eqn:g-state-space}
\begin{aligned}
E\dot{x}(t) = Ax(t)+Bu(t), \\
y(t) = Cx(t)+Du(t).
\end{aligned}
\end{equation}

If the eigenvalues of the matrix pair $(A,E)$ satisfy $\lambda_i +\bar{\lambda}_j \neq 0, \forall  i,j=1,2, \cdots ,m$, that ensures the solution $X$ of the CARE (\ref{eqn:GCARE}) exists and unique. The orthogonal projector $V\in \mathbb{R}^{n \times m}$ spanned by the $m$-dimensional rational Krylov subspace for a set of shift parameters $\mu_i\in \mathbb{C}^+; i=1,2, \cdots ,m$ is defined as
\small{\begin{align*}
\mathcal{K}_m=span \left(C^T, {(A^T-\mu_1 E^T)}^{-1}C^T,\cdots, \prod_{i=1}^{m} {(A^T-\mu_i E^T)}^{-1} C^T \right).
\end{align*} }

If  $\theta_j$ are the eigenvalues of $(V^T AV, V^T EV)$; $\mathbb{S}_m \in \mathbb{C}^+$ approximates the mirror eigenspace of $A-B(B^T XE)$ and $\delta \mathbb{S}_m$ is its border, the shifts are computed from
\begin{align*}
\mu_{i+1} = arg \left({\mathop{\mathrm{max}}_{\mu \in \delta \mathbb{S}_m} \left|\frac{\prod_{j=1}^i {(\mu-\mu_j)}}{\prod_{j=1}^i {(\mu-\theta_j)}} \right|}\right).
\end{align*}

According to the Galerkin condition and after simplification by matrix algebra, a low-rank CARE can be obtained as 
\begin{equation} \label{eqn:LR-GCARE}
\begin{aligned}
\hat{A}^T \hat{X}\hat{E} +\hat{E}^T \hat{X} \hat{A}-\hat{E}^T \hat{X} \hat{B}\hat{B}^T \hat{X}\hat{E}+ \hat{C}^T \hat{C}=0,
\end{aligned}
\end{equation}
where $\hat{X}=V^T XV, \hat{E}=V^T EV, \hat{A}=V^T AV, \hat{B}=V^T B$ and $\hat{C}=CV$. The equation (\ref{eqn:LR-GCARE}) is an approximated low-rank CARE and can be solved by any conventional method or MATLAB \texttt{care} command. Here $\hat{X}\in \mathbb{R}^{m \times m}$ is taken as low-rank approximation of $X$, corresponding to the low-rank CARE (\ref{eqn:LR-GCARE}). Then residual of the $m$-th iteration is
\begin{align*}
\|{\mathcal{R}_m}\|_F = \|{SJS^T}\|_F; \quad J&= \begin{bmatrix} 0 & 1 & 0 \\ 1 & 0 & 1 \\ 0 & 1 & 0 \end{bmatrix},
\end{align*}
where ${\|{.}\|}_F$ denotes the Frobenius norm and $S$ is a block upper triangular matrix in the QR factorization of the matrix $U$ derived as 
\begin{align*}
U ={ \begin{bmatrix} v_{m+1} \mu_{m+1} \\  E^T V_m \hat{X} H_m^{-T} e_m h_{m+1,m}^T  \\  -(I_n - V_m V_m^T) A^T v_{m+1} 
\end{bmatrix} }^T,
\end{align*}
where $H_m$ is a block upper Hessenberg matrix and $e_m$ is the matrix formed by the last $p$ columns of the $mp$-order identity matrix. For  $C^T = Q_0 R_0$ such that $R_0 = \beta_0$, the relative-residual can be estimated as
\begin{align*}
{\|{\mathcal{R}_m}\|}_{(relative)} = \frac{\|{\mathcal{R}_m\|}_F}{\|{\beta_0^T \beta_0}\|_F}.  
\end{align*}

Through RKSM the low-rank factor $Z$ of the approximate solution $X$ of the CARE (\ref{eqn:GCARE}) needs to be estimated, such that $X \approx Z Z^T $. The low-rank solution $\hat{X}$ is symmetric positive definite and the original solution $X$ can be approximated as $X=V \hat{X} V^T$. By the eigenvalue decomposition to the approximate solution $\hat{X}$ and truncating the negligible eigenvalues, the possible lowest order factor $Z$ of $X$ can be estimated as
\begin{align*}
X &= V \hat{X} V^T = V (T\Lambda T^T) V^T \\ &= V \begin{bmatrix} T_1 & T_2 \end{bmatrix} \begin{bmatrix} \Lambda_1 & 0 \\ 0 & \Lambda_2\end{bmatrix} \begin{bmatrix} T_1^T \\ T_2^T \end{bmatrix} V^T \\ &= VT_1 \Lambda_1 T_1^T V^T =  (VT_1 \Lambda_1^{\frac{1}{2}})(VT_1 \Lambda_1^{\frac{1}{2}})^T \\ &= ZZ^T.
\end{align*}

Here, $\Lambda_2$ consists the negligible eigenvalues. Summary of the above process is given in the Algorithm-\ref{pre:RKSM.}.
\begin{algorithm}[t]
\SetAlgoLined
\SetKwInOut{Input}{Input}
\SetKwInOut{Output}{Output}
\caption{Generalized RKSM.}
\label{pre:RKSM.}
\Input{ $E, A, B, C, i_{max}$ (number of iterations) and $\mu_i$ (initial shifts).}
\Output{Low-rank factored solution $Z$ such that $X\approx ZZ^T$.}  
Compute $Q_0 R_0 = C^T$ (QR factorization).\\
Choose $V_0 = Q_0$.\\  
\While{{not converged \ or \ $m \leq i_{\text{max}}$}}{%
Solve $v_{m+1} = {(A^T-\mu_{m+1} E^T)}^{-1}V_m$.\\
Compute shift for the next iteration.\\
Using Arnoldi algorithm orthogonalize $v_{m+1}$ against $V_m$ to obtain $\hat{v}_{m+1}$, such that $V_{m+1}=\begin{bmatrix} V_m, \hat{v}_{m+1} \end{bmatrix}$.\\
Assuming $\hat{E}=V_{m+1}^T EV_{m+1}$, $\hat{A}=V_{m+1}^T AV_{m+1}$, $\hat{B}=V_{m+1}^TB$ and $\hat{C}=CV_{m+1}$, for $\hat{X}$ solve the reduced-order Riccati equation $\hat{A}^T \hat{X} \hat{E}+ \hat{E}^T \hat{X} \hat{A}- \hat{E}^T \hat{X} \hat{B} \hat{B}^T \hat{X} \hat{E} + \hat{C}^T \hat{C}=0$.\\
Compute ${\|{\mathcal{R}_m}\|}_{(\text{relative})}$ for convergence.\\
}
Compute eigenvalue decomposition $\hat{X} = T\Lambda T^T=\begin{bmatrix} T_1 & T_2 \end{bmatrix} \begin{bmatrix} \Lambda_1 & 0 \\ 0 & \Lambda_2\end{bmatrix} \begin{bmatrix} T_1^T \\ T_2^T \end{bmatrix}$.\\
For negligible eigenvalues truncate $\Lambda_2$ and compute $Z=V_{m+1} T_1 \Lambda_1^{\frac{1}{2}}$. \\
\end{algorithm}

\subsection{Modified LRCF-ADI method for solving Lyapunov equations}

The generalized Continuous-time Algebraic Lyapunov Equation (CALE) associated with the LTI continuous-time system (\ref{eqn:g-state-space}) is 
\begin{equation} \label{eqn:GCALE}
\begin{aligned}
A^T XE+ E^TXA =- C^T C.
\end{aligned}
\end{equation}

The shift parameters $\mu_i \in \mathbb{C}^-$ are allowed and the initial iteration is taken as $X_0={X_0}^T \in \mathbb{R}^{n\times n}$. Assume $Z_i \in \mathbb{R}^{n\times ip}$ as the low-rank Cholesky factor of $X_i \in \mathbb{R}^{n\times n}$ such that $X_i=Z_i Z_i^T$ \cite{li2000model}. The ADI algorithm can be formed in terms of Cholesky factor $Z_i$ of $X_i$ and there will be no need to estimate or store $X_i$ at each iteration as only $Z_i$ is required \cite{penzl1999cyclic}. 

Considering $\gamma_i =\sqrt{-2\Rel(\mu_i)}$, the conventional low-rank Cholesky factor ADI iterations yield the form as
\begin{equation} \label{eqn:ADI-factor4}
\begin{aligned}
V_1 &= (A^T+\mu_1 E^T)^{-1}C^T,\\
Z_1 &=\gamma_1 V_1=\sqrt{-2\Rel(\mu_i)} (A^T+\mu_1 E^T)^{-1}C^T,\\
V_i &=H_{i-1,i} V_{i-1}=V_{i-1}-(\mu_i + \bar{\mu}_{i-1})(A^T+\mu_i E^T)^{-1} E^T V_{i-1},\\
Z_i &=\begin{bmatrix} Z_{i-1} \quad \gamma_i V_i \end{bmatrix}=\begin{bmatrix} Z_{i-1} \quad \sqrt{-2\Rel(\mu_i)} V_i \end{bmatrix}.
\end{aligned}
\end{equation}

Assume a set of adjustable shift parameters, for two subsequent block iterates $V_i, V_{i+1}$  of the ADI technique related to the pair of complex conjugated shifts $\{ \mu_i, \mu_{i+1}:=\bar{\mu}_i\}$ it holds 
\begin{equation} \label{eqn:t1}
\begin{aligned}
V_{i+1}=\bar{V}_i + 2\delta_i \text{Im}(V_i),
\end{aligned}
\end{equation}
where $\bar{(.)}$ indicates the complex conjugate with $\delta_i=\frac{\text{Re}(\mu_i)}{\text{Im}(\mu_i)}$ and iterates associated to real shifts are always purely real \cite{benner2013efficient}.

Then the following matrix for the basis extension can be obtained
\begin{equation} \label{eqn:iterations}
\begin{aligned}
\begin{bmatrix} V_i \ V_{i+1} \end{bmatrix}=\begin{bmatrix} \gamma_i (\Rel(V_i)+\delta_i \Imag(V_i)) \quad \gamma_i \sqrt{\delta_i^2 + 1} \Imag(V_i)\end{bmatrix}. 
\end{aligned}
\end{equation}

Then, for a pair of complex conjugate shifts at any iteration, the low-rank factor $Z_i$ can be computed as
\begin{equation} \label{eqn:ADI-factor5}
\begin{aligned}
Z_{i+1}=\begin{bmatrix}Z_{i-1} \ \gamma_i (\Rel(V_i)+\delta_i \Imag(V_i)) \quad \gamma_i \sqrt{\delta_i^2 + 1} \Imag(V_i)\end{bmatrix}. 
\end{aligned}
\end{equation}

The modified techniques discussed above is summarized in Algorithm-(\ref{pre:G-LRCF-ADI2}).
\begin{algorithm}[t]
\SetAlgoLined
\SetKwInOut{Input}{Input}
\SetKwInOut{Output}{Output}
\caption{Modified LRCF-ADI.}
\label{pre:G-LRCF-ADI2}
\Input{ $E, A, C,  i_{max}$ (number of iterations) and shift parameters $\{\mu_j\}_{j=1}^{i_{max}}$.}
\Output{Low-rank Cholesky-factor $Z$ such that $X\approx ZZ^T$.}  
Consider $Z_0=[ \ ]$.\\
\For {$i \gets 1$ to $i_{max}$} {
\eIf{$i=1$}{Solve $V_1=(A^T+\mu_1 E^T)^{-1} C^T$.}
{Compute $V_i=V_{i-1}-(\mu_i + \bar{\mu}_{i-1})(A^T+\mu_i E^T)^{-1} E^T V_{i-1}$.}
\eIf{$\text{Im}(\mu_i)=0$}{Update $Z_i=\begin{bmatrix} Z_{i-1} \quad \sqrt{-2\mu_i} V_i \end{bmatrix}$.}
{Assume $\gamma_i =\sqrt{-2\text{Re}(\mu_i)}$,  $\delta_i=\frac{\text{Re}(\mu_i)}{\text{Im}(\mu_i)}$,\\
Update $Z_{i+1}=\begin{bmatrix}Z_{i-1} \quad \gamma_i (\text{Re}(V_i)+\delta_i \text{Im}(V_i)) \quad \gamma_i \sqrt{\delta_i^2 + 1} \text{Im}(V_i) \end{bmatrix}$,\\
Compute $V_{i+1}=\bar{V}_i + 2\delta_i \ \text{Im}(V_i)$.\\
$i=i+1$
}
}
\end{algorithm}

\subsection{Basic Kleinman-Newton method}

For the system (\ref{eqn:g-state-space}), the residual $\mathcal{R}(X)$ of the CARE (\ref{eqn:GCARE}) is 
\begin{equation} \label{eqn:n-res}
\begin{aligned}
\mathcal{R}(X)=A^T X E + E^T XA - E^T XBB^T X E + C^T C.
\end{aligned}
\end{equation}  

Form the Fre'chet derivative at $X=X_i$, we have 
\begin{equation} \label{eqn:def-res}
\begin{aligned}
\mathcal{R}'(X_i)=(A-BB^T X E)^T X_i E + E^T X_i (A-BB^T X E).
\end{aligned}
\end{equation} 

Consider the Newton iteration and apply (\ref{eqn:def-res}), we have
\begin{equation} \label{eqn:newton}
\begin{aligned}
\mathcal{R}'(\Delta X_i)+\mathcal{R}(X_i)=0.
\end{aligned}
\end{equation}

Now, put $\Delta X_i = X_{i+1} - X_i$ in (\ref{eqn:newton}). Then, after simplification we get 
\begin{equation} \label{eqn:n-care}
\begin{aligned}
& (A-BB^T X_i E)^T X_{i+1} E + E^T X_{i+1} (A-BB^T X_i E) \\ &= -C^T C-E^T X_i BB^T X_i E.
\end{aligned}
\end{equation} 

Then, assume $\tilde{A}_i=A-BB^T X_i E$ and $W_i= \begin{bmatrix}  C^T \quad  E^T X_i B \end{bmatrix}$, then equation (\ref{eqn:n-care}) reduces to a generalized Lyapunov equation such as
\begin{equation} \label{eqn:n-lyap}
\begin{aligned}
\tilde{A}_i^T X_{i+1} E + E^T X_{i+1} \tilde{A}_i = - W_i W_i^T.
\end{aligned}
\end{equation}

The generalized CALE (\ref{eqn:n-lyap}) can be solved for $X_{i+1}$ by any conventional method, such as Low-Rank Cholesky-Factor Alternative Direction Implicit (LRCF-ADI) method and the corresponding feedback matrix $K_{i+1}=B^T X_{i+1}E$ can be estimated. The whole mechanism is called the Kleinman-Newton method \cite{benner2014numerical} for solving generalized CARE. The summary of the method is given in Algorithm-\ref{pre:KN.}.
\begin{algorithm}[t]
\SetAlgoLined
\SetKwInOut{Input}{Input}
\SetKwInOut{Output}{Output}
\caption{Basic Kleinman-Newton.}
\label{pre:KN.}
\Input{$E,A,B,C$ and $X_0$ (initial assumption).}
\Output{Approximate solution $X$ and feedback matrix $K$.}  
\While{ $i \leq i_{\text{max}} $}{%
Compute $\tilde{A}_i=A-BB^T X_i E$ and $W_i= \begin{bmatrix} C^T \quad     E^T X_i B \end{bmatrix}$;\\
For $X_{i+1}$, solve $\tilde{A}_i^T X_{i+1} E + E^T X_{i+1} \tilde{A}_i = - W_i W_i^T$;\\
Compute $K_{i+1}=B^T X_{i+1} E$.
}  
\end{algorithm} 

\section{Solving Riccati equations arising from the index-1 descriptor systems}

In this section, we discuss the updated RKSM techniques for solving the Riccati equation derived from index-1 descriptor systems, it includes the stopping criteria, sparsity preservation, and estimation of the optimal feedback matrix. Then, LRCF-ADI based Kleinman-Newton method with the adjustment for index-1 descriptor systems and finding the optimal feedback matrix are discussed. Finally, the stabilized system is narrated accordingly.  

\subsection{Updated Rational Krylov subspace method}

Let us consider the LTI continuous-time system (\ref{eqn:d-state-space}) and introduce an orthogonal projector $V$ spanned by the $m$ dimensional rational Krylov subspace for a set of given shift parameters $\mu_i\in \mathbb{C}^+; i=1,2,...,m$ is defined as
\small{\begin{align*}
\mathcal{K}_m=span \left(\mathcal{C}^T, {(\mathcal{A}^T-\mu_1 \mathcal{E}^T)}^{-1}\mathcal{C}^T,...,\prod_{i=1}^{m} {(\mathcal{A}^T-\mu_i \mathcal{E}^T)}^{-1} \mathcal{C}^T \right).
\end{align*} }

Again, consider the CARE (\ref{eqn:D-GCARE}) and apply the Galerkin condition on it. Then, after the simplification by matrix algebra, a low-rank CARE can be achieved as 

\begin{equation} \label{eqn:LR-DCARE}
\begin{aligned}
\hat{\mathcal{A}}^T \hat{X} \hat{\mathcal{E}}+\hat{\mathcal{E}}^T \hat{X} \hat{\mathcal{A}}-\hat{\mathcal{E}}^T \hat{X} \hat{\mathcal{B}} \hat{\mathcal{B}}^T \hat{X} \hat{\mathcal{E}}+\hat{\mathcal{C}}^T \hat{\mathcal{C}}=0,
\end{aligned}
\end{equation}
where, $\hat{X}=V^T XV, \hat{\mathcal{E}}=V^T \mathcal{E}V, \hat{\mathcal{A}}=V^T \mathcal{A}V, \hat{\mathcal{B}}=V^T \mathcal{B}$ and $\hat{\mathcal{C}}=\mathcal{C}V$. The equation (\ref{eqn:LR-DCARE}) is a low-rank CARE and can be solved by MATLAB \texttt{care} command or any existing methods, such as \texttt{Schur}-decomposition method. 

For the quick and smooth convergence of the proposed algorithm, adjustable shift selection is crucial and we are adopting the adaptive shift approach for index-1 descriptor systems \cite{benner2014self}. This process required to be recursive and in each step the subspace to all the projector $V$ generated with the current set of shifts will be extended.

\subsubsection{Stopping criteria and related theorem}

Arnoldi relation is a very essential tool for the computation of residual of the RKSM iterations. To avoid extra matrix-vector multiplies per iteration, the computation of projected matrix $T_m$ can be performed more efficiently than the explicit product $T_m = V_m^T \mathcal{A}V_m$. To find the stopping criteria, we have to consider following lemmas.

\begin{lemma} [\textbf{Arnoldi relation \cite{abidi2016adaptive}}] \label{a.relation}
Let $\mathcal{K}_m = \text{span} (V_m)$ be the rational Krylov subspace with the shift parameters $\mu_i\in \mathbb{C}^+; i=1,2, \cdots ,m$. Then $\mathcal{K}_m$ satisfies the Arnoldi relation as follows
\begin{equation} \label{eqn:a.relation1}
\begin{aligned}
\mathcal{A}^TV_m=V_m T_m + \hat{v}_{m+1} g_m^T \ ; \quad V_m^T V_m = I_m,
\end{aligned}
\end{equation}
where $\hat{v}_{m+1} \beta_1 = v_{m+1} \mu_{m+1} - (I_n - V_m V_m^T) \mathcal{A}^T v_{m+1}$ is the QR decomposition of the right hand side matrix with $g_m^T = \beta_1 h_{m+1,m} e_m^T H_m^{-1}$. 
\end{lemma}

\begin{lemma} [\textbf{Building the projected matrix \cite{lin2013minimal}}] \label{projected}
Let the column vectors of $V_m$ be an orthonormal basis of the rational Krylov subspace $\mathcal{K}_m$ with the block diagonal matrix $D_m = \text{diag} (\mu_1 \mathcal{E}^T, \mu_2 \mathcal{E}^T, \cdots , \mu_m \mathcal{E}^T)$, where $\{ \mu_1, \mu_2, \cdots , \mu_m \}$ is the set of shift parameters used in the algorithm. Then for the projected matrix $T_m = V_m^T \mathcal{A} V_m$ the following relation holds
\begin{equation} \label{eqn:projected1}
\begin{aligned}
T_m = \left[ I_m + V_m^T \mathcal{E}^T V_m H_m D_m - V_m^T \mathcal{A}^T v_{m+1} h_{m+1,m} e_m^T \right] H_m^{-1}.
\end{aligned}
\end{equation}
\end{lemma}

\begin{theorem}  [\textbf{Residual of the RKSM iterations}] \label{Trksm}
Let $V_m$ be the orthogonal projector spanned by the rational Krylov subspace $\mathcal{K}_m$ and $X\approx V\hat{X} V^T$ is the solution of the CARE (\ref{eqn:D-GCARE}) using the low-rank solution $\hat{X}$. Then, the residual of $m$-th iteration can be computed as 
\begin{equation} \label{eqn:trksm1}
\begin{aligned}
{\|{\mathcal{R}_m}\|}_F={\|{ S J S^T}\|}_F \ ; \quad J=\begin{bmatrix} 0 & I & 0 \\ I & 0 & I \\  0 & I & 0 \end{bmatrix},
\end{aligned}
\end{equation}
where ${\|{.}\|}_F$ denotes the Frobenius norm and $S$ is a block upper triangular matrix in the QR factorization of the matrix $U$ is defined as 
\begin{equation} \label{eqn:trksm2}
\begin{aligned}
U = \begin{bmatrix} v_{m+1} \mu_{m+1} \quad \mathcal{E}^T V_m \hat{X} H_m^{-T} e_m h_{m+1,m}^T \quad -(I_n - V_m V_m^T) \mathcal{A}^T v_{m+1} \end{bmatrix}.
\end{aligned}
\end{equation}
\end{theorem}

\begin{proof}
Assume $f = (I_n - V_m V_m^T) \mathcal{A}^T v_{m+1}$ and consider the reduced QR factorization $\mathcal{C}^T=V_0 \beta_0$. Then, by putting the relations in equation (\ref{eqn:a.relation1}) of Lemma-(\ref{a.relation}), the relation can be written as
\begin{equation} \label{eqn:trksm3}
\begin{aligned}
\mathcal{A}^TV_m &= V_m T_m + {v}_{m+1} \mu_{m+1} h_{m+1,m} e_m^T H_m^{-1} -  g h_{m+1,m} e_m^T H_m^{-1}\\
&= V_m T_m + ({v}_{m+1} \mu_{m+1} - f) h_{m+1,m} e_m^T H_m^{-1}.
\end{aligned}
\end{equation}

The residual of the CARE (\ref{eqn:D-GCARE}) can be written as
\begin{equation} \label{eqn:trksm4}
\begin{aligned}
\mathcal{R}(X) = \mathcal{A}^T X\mathcal{E} + \mathcal{E}^T X\mathcal{A} - \mathcal{E}^T X\mathcal{B} \mathcal{B}^T X\mathcal{E} + \mathcal{C}^T \mathcal{C}.
\end{aligned}
\end{equation}
	
Consider the approximate solution using the low-rank solution $\hat{X}$ as $X= V_m \hat{X} V_m^T$ and equation-(\ref{eqn:projected1}) in Lemma-(\ref{projected}), then applying (\ref{eqn:trksm3}) in (\ref{eqn:trksm4}), we get
\begin{equation} \label{eqn:trksm5}
\begin{aligned}
&\mathcal{R}(X_m) = V_m T_m \hat{X}_m V_m^T \mathcal{E} + ({v}_{m+1} \mu_{m+1} - f) h_{m+1,m} e_m^T H_m^{-1} \hat{X}_m V_m^T \mathcal{E} + \mathcal{E}^T V_m \hat{X} T_m^T V_m^T \\ &+ \mathcal{E}^T V_m \hat{X} H_m^{-T} e_m h_{m+1,m}^T ({v}_{m+1} \mu_{m+1} - f)^T - \mathcal{E}^T V_m \hat{X} V_m^T \mathcal{B} \mathcal{B}^T V_m \hat{X} V_m^T \mathcal{E} + \mathcal{C}^T \mathcal{C}\\  
&= ({v}_{m+1} \mu_{m+1} - f) h_{m+1,m} e_m^T H_m^{-1} \hat{X}_m V_m^T \mathcal{E} + \mathcal{E}^T V_m \hat{X} H_m^{-T} e_m h_{m+1,m}^T ({v}_{m+1} \mu_{m+1} - f)^T \\ &+ \mathcal{M}^T X \mathcal{E} + \mathcal{E}^T X \mathcal{M} - \mathcal{E}^T X \mathcal{B} \mathcal{B}^T X \mathcal{E} + \mathcal{C}^T \mathcal{C} \ ; \ \mathcal{M} = \mathcal{A}^T V_m V_m^T\\
&= \mathcal{E}^T V_m \hat{X} H_m^{-T} e_m h_{m+1,m}^T \mu_{m+1}^T v_{m+1}^T + ({v}_{m+1} \mu_{m+1} - f) h_{m+1,m} e_m^T H_m^{-1} \hat{X}_m V_m^T \mathcal{E} \\ &- \mathcal{E}^T V_m \hat{X} H_m^{-T} e_m h_{m+1,m}^T f^T + 0 \\
&=\begin{bmatrix} v_{m+1} \mu_{m+1} \ \ \mathcal{E}^T V_m \hat{X}_m H_m^{-T} e_m h_{m+1,m}^T \ \ -f \end{bmatrix}
\begin{bmatrix} 0 & I & 0 \\ I & 0 & I \\ 0 & I & 0 \end{bmatrix}
\begin{bmatrix} \mu_{m+1}^T v_{m+1}^T \\  h_{m+1,m} e_m^T H_m^{-1} \hat{X}_m^T V_m^T \mathcal{E} \\ -f^T \end{bmatrix}\\
&= S J S^T.
\end{aligned}
\end{equation}

Thus, the proof follows from the Frobenius norm of the equation (\ref{eqn:trksm5}).

\end{proof}

Then the relative-residual with respect to $\beta_0$ can be estimated as follows 
\begin{equation} \label{eqn:relres}
\begin{aligned}
{\|{\mathcal{R}_m}\|}_{(\text{relative})} = \frac{\|{\mathcal{R}_m}\|_F}{\|{\beta_0^T \beta_0}\|_F}.
\end{aligned}
\end{equation}

\subsubsection{Sparsity preservation}

The matrix $\mathcal{A}$ in (\ref{eqn:d-state-space}) is in dense form, which is contradictory to the aim of the work and the rate of convergence of the converted system is slow enough. So, to bypass these drawbacks at each iteration a shifted linear system needs to be solved for $v_i$ as

\begin{equation} \label{eqn:sparse_shifted}
\begin{aligned}
(\mathcal{A}^T-\mu_i \mathcal{E}^T)v_i &=V_{i-1},\\
or, {\begin{bmatrix}  J_1-{\mu_i}E_1  &  J_2 \\ J_3  &  J_4   \end{bmatrix}}^T \begin{bmatrix}  v_i  \\ \Gamma \end{bmatrix}
&=  \begin{bmatrix}    V_{i-1} \\ 0  \end{bmatrix}.
\end{aligned}
\end{equation} 

Here $\Gamma = - J_4^{-1} J_3 v_i$ is the truncated term. The linear system (\ref{eqn:sparse_shifted}) is higher dimensional but sparse and can be solved by the conventional sparse-direct solvers very efficiently \cite{uddin2019computational}. To improve the consistency of the RKSM approach, explicit form of the reduced-order matrices must not be used to construct reduced-order system. The sparsity-preserving reduced-order matrices can be attained by following way
\begin{equation} \label{eqn:sprom}
\begin{aligned}
\hat{\mathcal{E}} &= V^T E_1 V,\quad
\hat{\mathcal{A}} = V^T J_1 V - (V^T J_2){J_4}^{-1}(J_3 V),\\
\hat{\mathcal{B}} &= V^T B_1 - (V^T J_2) {J_4}^{-1} B_2,\quad
\hat{\mathcal{C}} = C_1 V - C_2 {J_4}^{-1} (J_3 V).
\end{aligned}
\end{equation} 

\subsubsection{Treatment for the unstable systems}

If the system is unstable, a Bernoulli stabilization is required through an initial-feedback matrix $K_0$ to estimate $\mathcal{A}_f = \mathcal{A}-\mathcal{B} K_0$ and the matrix $\mathcal{A}$ needs to be replaced \cite{benner2017model}. Then, the system (\ref{eqn:d-state-space}) and CARE (\ref{eqn:D-GCARE}) need to be re-defined as
\begin{equation} \label{eqn:re-state-space}
\begin{aligned}
\mathcal{E}\dot{x}(t)&=\mathcal{A}_f x(t)+\mathcal{B}u(t),\\
y(t)&=\mathcal{C}x(t)+\mathcal{D}u(t),
\end{aligned}
\end{equation}
\begin{equation} \label{eqn:re-GCARE}
\begin{aligned}
\mathcal{A}_f^T X\mathcal{E}+\mathcal{E}^T X\mathcal{A}_f-\mathcal{E}^T X\mathcal{B} \mathcal{B}^T X\mathcal{E}+\mathcal{C}^T \mathcal{C}=0.
\end{aligned}
\end{equation}  

Then, for every iterations $K=\hat{\mathcal{B}}^T \hat{X}V^T \mathcal{E}$ needs to be updated by the solution $\hat{X}$ of (\ref{eqn:LR-GCARE}) and the rational Krylov subspace for the projector $V$ needs to be redefined as
\small{\begin{align*}
\mathcal{K}_m=span \left(\mathcal{C}^T, {(\mathcal{A}_f^T-\mu_1 \mathcal{E}^T)}^{-1}\mathcal{C}^T,...,\prod_{i=1}^{m} {(\mathcal{A}_f^T-\mu_i \mathcal{E}^T)}^{-1} \mathcal{C}^T \right).
\end{align*} }

For the stabilized system using the initial-feedback matrix $K_0$, the expressions (\ref{eqn:sparse_shifted}) can be written as 
\begin{equation} \label{eqn:sparse_shifted_st}
\begin{aligned}
(\mathcal{A}_f^T-\mu_i \mathcal{E}^T)v_i &=V_{i-1},\\
or, {\begin{bmatrix}  (J_1-B_1 K_0)-{\mu_i}E_1  &  J_2 \\ J_3-B_2 K_0  &  J_4   \end{bmatrix}}^T \begin{bmatrix}  v_i  \\ *\end{bmatrix} &=  \begin{bmatrix} V_{i-1} \\ 0  \end{bmatrix}.
\end{aligned}
\end{equation} 

To evaluate the shifted linear systems, explicit inversion of $\mathcal{A} - \mathcal{B} K$ should be avoided in practice, instead the \textit{Sherman-Morrison-Woodbury} formula needs to be used as follows
\begin{align*}
(\mathcal{A} - \mathcal{B} K)^{-1} = \mathcal{A}^{-1} + \mathcal{A}^{-1} \mathcal{B}(I- K\mathcal{A}^{-1} \mathcal{B})^{-1}K \mathcal{A}^{-1}.
\end{align*}

\subsubsection{Estimation of the optimal feedback matrix}

The low-rank solution $\hat{X}$ is symmetric and positive definite and can be factorized  as $\hat{X}=YY^T$. The original solution can be approximated as $X=V \hat{X} V^T=VY(VY)^T$. Finally, the desired low-rank factored solution $Z=VY$ of the CARE (\ref{eqn:D-GCARE}) will be stored and the optimal feedback matrix $K^o = \mathcal{B}^T X \mathcal{E} = \mathcal{B}^T (ZZ^T) \mathcal{E}$ can be estimated. This process is iterative and will continue until the desired convergence is achieved. The whole process is summarized in the Algorithm-\ref{pre:RKSM-FB.}.
\begin{algorithm}[t]
\SetAlgoLined
\SetKwInOut{Input}{Input}
\SetKwInOut{Output}{Output}
\caption{Updated RKSM (sparsity-preserving).}
\label{pre:RKSM-FB.}
\Input{ $E_1, J_1, J_2, J_3, J_4, B_1, B_2, C_1, C_2, K_0$ (initial feedback matrix) $i_{max}$ (number of iterations) and $\mu_i$ (initial shifts).}
\Output{Low-rank factored solution $Z$ such that $X\approx ZZ^T$ and optimal feedback matrix $K^o$.}  
Compute $Q_0 R_0 = (C_1 - C_2{J_4}^{-1}J_3)^T$ (QR factorization).\\ 
Choose $V_0 = Q_0$.\\
Choose $K=K_0$.\\
\While{{not converged \ or \ $m \leq i_{\text{max}}$}}{%
Solve the linear system (\ref{eqn:sparse_shifted_st}) for $v_{m+1}$.\\
Compute adaptive shifts for the next iterations (if store is empty).\\
Using Arnoldi algorithm orthogonalize $v_{m+1}$ against $V_m$ to obtain $\hat{v}_{m+1}$, such that $V_{m+1}=\begin{bmatrix} V_m, \hat{v}_{m+1} \end{bmatrix}$.\\ 
Assuming $\hat{\mathcal{E}}$, $\hat{\mathcal{A}}$, $\hat{\mathcal{B}}$ and $\hat{\mathcal{C}}$ are defined in (\ref{eqn:sprom}), for $\hat{X}$ solve the reduced-order Riccati equation
$\hat{\mathcal{A}}^T \hat{X}\hat{\mathcal{E}}+ \hat{\mathcal{E}}^T \hat{X} \hat{\mathcal{A}}-\hat{\mathcal{E}}^T \hat{X} \hat{\mathcal{B}} \hat{\mathcal{B}}^T \hat{X} \hat{\mathcal{E}} + \hat{\mathcal{C}}^T \hat{\mathcal{C}}=0$.\\
Update $K=(\hat{\mathcal{B}}^T \hat{X})V_{m+1} E_1$.\\
Compute ${\|{\mathcal{R}_m}\|}_{(\text{relative})}$ for convergence.\\
}  
Compute eigenvalue decomposition 
$\hat{X} = T\Lambda T^T=\begin{bmatrix} T_1 & T_2 \end{bmatrix}
\begin{bmatrix} \Lambda_1 & 0 \\ 0 & \Lambda_2\end{bmatrix}
\begin{bmatrix} T_1^T \\ T_2^T \end{bmatrix}$.\\
For negligible eigenvalues truncate $\Lambda_2$ and construct
$Z=V_{m+1} T_1 \Lambda_1^{\frac{1}{2}}$.\\ 
Compute the optimal feedback matrix $K^o = (B_1 - J_2{J_4}^{-1}B_2)^T(ZZ^T)E_1$.
\end{algorithm}  

\subsection{Updated LRCF-ADI Based modified Kleinman-Newton method}

In each iteration of Algorithm-\ref{pre:KN.}, the generalized CALE needs to be solved for once and there are several techniques available to do it. In his Ph.D. thesis, Kuerschner discussed low-rank ADI approaches (Algorithm-3.2 chapter-3 and Algorithm-6.2 in chapter-6) for solving generalized CALE derived from generalized CARE in the iterative loops of Kleinman-Newton algorithm \cite{kurschner2016efficient}. Now, we need to implement above mechanisms for the index-1 descriptor system in the sparse form. For the adjustment, some modifications are required as given below.

\subsection{Convergence criteria and recurrence relations}
The computation of residuals of the ADI iterations can be achieved using the simplified technique implementing the adjustable shift parameters. This approach will be efficient for time-dealing and memory allocation. Following lemmas represent some important properties of the ADI iterates. 

\begin{lemma} [\textbf{Convergence of ADI iterates \cite{feitzinger2009inexact}}] \label{Lres1}
Let $X_i$ be the solution of the CALE (\ref{eqn:GCALE}) and consider $X_i^{(k)}$ be an iterate of the ADI method. Then for all $\mu_i \in \mathbb{C}^-$ the relation holds
\begin{equation} \label{eqn:lyapres1}
\begin{aligned}
X_i^{(k+1)} - X_i = \left( \prod_{j=1}^{k}{\mathcal{A}_{k,\mu_j}} \right) (X_l^{(0)} - X_l) \left( \prod_{j=1}^{k}{\mathcal{A}_{k,\mu_j}} \right)^T,
\end{aligned}
\end{equation} 
where $\mathcal{A}_{k,\mu_j} = (\mathcal{A}^T-\bar{\mu_j} \mathcal{E}^T)(\mathcal{A}^T+\mu_j \mathcal{E}^T)^{-1}$.
\end{lemma}

\begin{lemma} [\textbf{Residual of the ADI iterates \cite{feitzinger2009inexact}}] \label{Lres2}
Let $X_i^{(k)}$ be an iterate of the CALE (\ref{eqn:GCALE}) by the ADI method. Then considering $\mathcal{A}_{k,\mu_j} = (\mathcal{A}^T-\bar{\mu_j} \mathcal{E}^T)(\mathcal{A}^T+\mu_j \mathcal{E}^T)^{-1}$ for all $\mu_i \in \mathbb{C}^-$ the residuals at $X_i^{(k)}$ have the form
\begin{equation} \label{eqn:lyapres2}
\begin{aligned}
\mathcal{R}_i^{(k+1)} &= \mathcal{E}^T X_i^{(k)} \mathcal{A} + \mathcal{A}^T X_i^{(k)} \mathcal{E} + \mathcal{C}^T \mathcal{C} \\ 
&= \left( \prod_{j=1}^{k}{\mathcal{A}_{k,\mu_j}} \right) (\mathcal{E}^T X_i^{(0)} \mathcal{A} + \mathcal{A}^T X_i^{(0)} \mathcal{E} + \mathcal{C}^T \mathcal{C}) \left( \prod_{j=1}^{k}{\mathcal{A}_{k,\mu_j}} \right)^T.
\end{aligned}
\end{equation} 
\end{lemma}

\begin{theorem} [\textbf{Residual factor of the ADI iterations}] \label{Tadi2}
The residual for CALE (\ref{eqn:GCALE}) at $i$-th iteration of ADI method is of rank at most $m$ and is given by 
\begin{equation} \label{eqn:residue1}
\begin{aligned}
\mathcal{R}(X_i) = \mathcal{A}^T X_i \mathcal{E}+\mathcal{E}^T X_i \mathcal{A}+\mathcal{C}^T \mathcal{C} = \mathcal{W}_i \mathcal{W}_i^T, 
\end{aligned}
\end{equation}
with $\mathcal{W}_i=\left( \prod_{i=1}^{m} {(\mathcal{A}^T-\bar{\mu_i} \mathcal{E}^T)(\mathcal{A}^T+\mu_i \mathcal{E}^T)^{-1}} \right)\mathcal{C}^T$, for all $\mu_i \in \mathbb{C}^-$.
\end{theorem}

\begin{proof}
Consider $X$ is the solution of the CALE (\ref{eqn:GCALE}) and $X_i$ is its $i$-th iterate by ADI method. Then the residual of the $i$-th iteration in terms of $X_i$ can be written as
\begin{equation} \label{eqn:residue2}
\begin{aligned}
\mathcal{R}(X_i) = \mathcal{A}^T X_i \mathcal{E} + \mathcal{E}^T X_i \mathcal{A} + \mathcal{C}^T \mathcal{C}. 
\end{aligned}
\end{equation}
	
For all $\mu_i \in \mathbb{C}^-$, the Stein's equation is equivalent to the CALE (\ref{eqn:GCALE}) and it can be written as
\begin{equation} \label{eqn:Stein}
\begin{aligned}
X &= {(\mathcal{A}^T-\bar{\mu_i} \mathcal{E}^T)(\mathcal{A}^T+\mu_i \mathcal{E}^T)^{-1}} X_i {(\mathcal{A}^T-\bar{\mu_i} \mathcal{E}^T)^T (\mathcal{A}^T+\mu_i \mathcal{E}^T)^{-T}} \\ &- 2\mu_i (\mathcal{A}^T+\mu_i \mathcal{E}^T)^{-T} \mathcal{C}^T \mathcal{C} (\mathcal{A}^T+\mu_i \mathcal{E}^T)^{-1}
\end{aligned}
\end{equation}
	
Using the Stein's equation (\ref{eqn:Stein}) in Lemma-\ref{Lres1}, we can find some initial iterates $X_i^{(0)} = 0$ for which the residuals $\mathcal{R}_i^{(k+1)} \ge 0$. Then according to the Lemma-\ref{Lres2} the residual (\ref{eqn:residue2}) can be defined as
\small{
\begin{equation} \label{eqn:residue3}
\begin{aligned}
\mathcal{R}(X_i)& =\left( \prod_{i=1}^{m} {(\mathcal{A}^T-\bar{\mu_i} \mathcal{E}^T)(\mathcal{A}^T+\mu_i \mathcal{E}^T)^{-1}} \right)\mathcal{C}^T
\mathcal{C}\left( \prod_{i=1}^{m} {(\mathcal{A}^T-\bar{\mu_i} \mathcal{E}^T)(\mathcal{A}^T+\mu_i \mathcal{E}^T)^{-1}} \right)^T,\\
&=\mathcal{W}_i \mathcal{W}_i^T,
\end{aligned}
\end{equation}
}
with $\mathcal{W}_i=\left( \prod_{i=1}^{m} {(\mathcal{A}^T-\bar{\mu_i} \mathcal{E}^T)(\mathcal{A}^T+\mu_i \mathcal{E}^T)^{-1}} \right)\mathcal{C}^T$. 
\end{proof}

Thus, for $\mathcal{W}_i \mathcal{W}_i^T \leq \tau$ the ADI iterations in the LRCF-ADI algorithm needs to be stopped, where $\tau$ is a given margin of tolerance. Then with the residual factor relation for $V_i$ can be derived as
\begin{equation} \label{eqn:relation1}
\begin{aligned}
V_i =(\mathcal{A}^T+\mu_i \mathcal{E}^T)^{-1} \mathcal{W}_{i-1}.
\end{aligned}
\end{equation}

Then, using (\ref{eqn:relation1}) the residual factor $W_i$ can be derived in a recursive form as
\begin{equation} \label{eqn:relation2}
\begin{aligned}
\mathcal{W}_i=\mathcal{W}_{i-1}-2\Rel(\mu_i) \mathcal{E}^T V_i.
\end{aligned}
\end{equation}

In case of real setting, $\mu_{i+1}:=\bar{\mu_i}$ needs to be considered to find the following form 
\begin{equation} \label{eqn:relation3}
\begin{aligned}
\mathcal{W}_{i+1}=\mathcal{W}_{i-1}+2\gamma_i^2 \mathcal{E}^T \left[ \Rel(V_i)+\delta_i \Imag(V_i) \right].
\end{aligned}
\end{equation}

The summary of above techniques is given in the Algorithm-(\ref{pre:LRCF-ADI}).
\begin{algorithm}[t]
\SetAlgoLined
\SetKwInOut{Input}{Input}
\SetKwInOut{Output}{Output}
\caption{Updated LRCF-ADI (Real Version).}
\label{pre:LRCF-ADI}
\Input{ $\mathcal{E}, \mathcal{A}, \mathcal{C}$, $\tau$ (tolerance), $ i_{max}$ (number of iterations) and shift parameters $\{\mu_j\}_{j=1}^{i_{max}}$.}
\Output{Low-rank Cholesky-factor $Z$ such that $X\approx ZZ^T$.}  
Consider $\mathcal{W}_0=\mathcal{C}^T, \  Z_0=[ \ ]$ and $i=1$.\\
\While{ $\|\mathcal{W}_{i-1} \mathcal{W}_{i-1}^T\| \geq \tau$ or $i \leq  i_{max}$}{%
Solve $V_i=(\mathcal{A}^T+\mu_i \mathcal{E}^T)^{-1} \mathcal{W}_{i-1}$.\\
\eIf{$\Imag(\mu_i)=0$}
{ Update $Z_i=\begin{bmatrix} Z_{i-1} \quad \sqrt{-2\mu_i} V_i \end{bmatrix}$,\\	
Compute $\mathcal{W}_i=\mathcal{W}_{i-1}-2\mu_i \mathcal{E}^T V_i$.}
{Assume $\gamma_i=\sqrt{-2\Rel(\mu_i)}$, \quad $\delta_i=\frac{\Rel(\mu_i)}{\Imag(\mu_i)}$,\\
Update $Z_{i+1}=\begin{bmatrix}Z_{i-1} \quad \gamma_i (\Rel(V_i)+\delta_i \Imag(V_i)) \quad \gamma_i \sqrt{\delta_i^2 + 1} \Imag(V_i)\end{bmatrix}$,\\
Compute $\mathcal{W}_{i+1}=\mathcal{W}_{i-1}+2\gamma_i^2 \mathcal{E}^T \left[ \Rel(V_i)+\delta_i \Imag(V_i) \right]$.\\
$i=i+1$
}
$i=i+1$
}
\end{algorithm} 

\subsubsection{Adjustment for the unstable systems}
For the unstable index-1 descriptor system, initial feedback matrix $K_0$ needs to be introduced and instead of $(\tilde{\mathcal{A}}^{(i)}, \mathcal{E})$, corresponding shift parameters are needed to be computed from eigen-pair $(\tilde{\mathcal{A}}^{(i)} - \mathcal{B}K_0, \mathcal{E})$. The sparse form of the eigen-pair can be structured as

\begin{equation} \label{eqn:n-eig}
\begin{aligned}
& (\tilde{\mathcal{A}}^{(i)} - \mathcal{B}K_0, \mathcal{E}) =((\mathcal{A}-\mathcal{B}K_0)-\mathcal{B}\mathcal{B}^T (Z^{(i)} (Z^{(i)})^T) \mathcal{E}, \mathcal{E}),\\
&=\left(\begin{bmatrix}  (J_1 - B_1 K_0) - B_1 B_1^T (Z^{(i)} (Z^{(i)})^T) E_1  &  \quad J_2 \\  ( J_3 -  B_2 K_0) - B_2 B_1^T (Z^{(i)} (Z^{(i)})^T) E_1  & \quad J_4 \end{bmatrix}, \begin{bmatrix}  E_1  &  0 \\   0  & 0\end{bmatrix}\right).
\end{aligned}
\end{equation}

To find $V_j^{(i)}$, in each ADI (inner) iteration a shifted linear system needs to be solved as  
\begin{equation} \label{eqn:n-linear-st}
\begin{aligned}
((\tilde{\mathcal{A}}^{(i)} -  \mathcal{B}K_0) +\mu_j^{(i)} \mathcal{E})^TV_j^{(i)} &=\mathcal{W}_{j-1}^{(i)},\\
or,((\mathcal{A} - \mathcal{B}K_0) - \mathcal{B}\mathcal{B}^T (Z_{j-1}^{(i)} (Z_{j-1}^{(i)})^T) \mathcal{E} +\mu_j^{(i)} \mathcal{E})^TV_j^{(i)} &=\mathcal{W}_{j-1}^{(i)}.
\end{aligned}
\end{equation}

Thus, $V_j^{(i)}$ can be obtained from the sparse form of the shifted linear system structured as

\small{\begin{equation} \label{eqn:n-broken-st}
\begin{aligned}
\begin{bmatrix}  (J_1  - B_1 K_0 ) - B_1 B_1^T (Z_{j-1}^{(i)} (Z_{j-1}^{(i)})^T) E_1 + \mu_j E_1  &  \quad J_2 \\   (J_3  -  B_2 K_0) - B_2 B_1^T (Z_{j-1}^{(i)} (Z_{j-1}^{(i)})^T) E_1  & \quad J_4 \end{bmatrix}^T \begin{bmatrix}  V_j^{(i)} \\  \Gamma \end{bmatrix}\\ =\begin{bmatrix}  C_1^T  &  E_1^T (Z_{j-1}^{(i)} & (Z_{j-1}^{(i)})^T) B_1 \\  C_2^T   & 0\end{bmatrix}.
\end{aligned}
\end{equation}}

\subsubsection{Estimation of the optimal feedback matrix}

The feedback matrix $K^{(i)}=\mathcal{B}^T X^{(i)} \mathcal{E}=\mathcal{B}^T (Z^{(i)}(Z^{(i)})^T) \mathcal{E}$ needs to be computed in each ADI (inner) iteration and the optimal feedback matrix $K^o=K^{(i_{max})}$ needs to be stored after the final Newton (outer) iteration. The summary of the modified method is given in the Algorithm-(\ref{pre:KN-LRCF-ADI.}).
\begin{algorithm}[]
\SetAlgoLined
\SetKwInOut{Input}{Input}
\SetKwInOut{Output}{Output}
\caption{Modified KN-LRCF-ADI (sparsity-preserving).}
\label{pre:KN-LRCF-ADI.}
\Input{ $E_1, J_1, J_2, J_3, J_4, B_1, B_2, C_1, C_2, K_0$ (initial feedback matrix), and $\tau$ (tolerance).}
\Output{Low-rank Cholesky-factor $Z$ such that $X \approx ZZ^T$ and optimal feedback matrix $K^o$.}
\For {$i \gets 1$ to $i_{max}$}{
Choose $Z_0^{(i)}=[ \  ], $ $K_0^{(i)}=K_0$ and $j=0$.\\
Assume $\mathcal{W}_0^{(i)}= \begin{bmatrix}  C_1^T &  (K^{(i-1)})^T \\ C_2^T & 0 \end{bmatrix}$.\\
Compute adaptive shifts $\left \{  \mu_1^{(i)},......,\mu_J^{(i)} \right \}$ from the eigenpair defined in (\ref{eqn:n-eig}).\\ 
\While{ $\left(\|\mathcal{W}_j^{(i)}\|^2> \tau \|\mathcal{W}_0^{(i)}\|^2 \right)$}{%
$j=j+1$\\
Solve the linear system (\ref{eqn:n-broken-st}) for $V_j^{(i)}$.\\
\eIf{$\text{Im}(\mu_j^{(i)})=0$}
{Update $Z_j^{(i)}=\begin{bmatrix} Z_{j-1}^{(i)} \quad & \sqrt{-2 \mu_i} V_j^{(i)} \end{bmatrix}$,\\
Compute $\mathcal{W}_j^{(i)}=\mathcal{W}_{j-1}^{(i)}-2 \mu_j^{(i)} E_1^T V_j^{(i)}$,\\
Compute $K_j^{(i)}=K_{j-1}^{(i)}-2 \mu_j^{(i)} (B_1 - J_2{J_4}^{-1}B_2)^T V_j^{(i)} (V_j^{(i)})^T E_1$.}
{Assume $\gamma_j^{(i)}=\sqrt{-2\text{Re}(\mu_j^{(i)})}$, \quad $\beta_j^{(i)}=\frac{\text{Re}(\mu_j^{(i)})}{\text{Im}(\mu_j^{(i)})}$,
\quad $\delta_j^{(i)} = \text{Re}(V_j^{(i)}) + \beta_j^{(i)}\text{Im}(V_j^{(i)}),$\\
Compute $ Z_d=\begin{bmatrix} \gamma_j^{(i)} \delta_j^{(i)} \quad & \gamma_j^{(i)} \sqrt{(\beta_j^{(i)})^2+1} \ \text{Im}(\mu_j^{(i)}) \end{bmatrix}$,\\
Update $Z_{j+1}^{(i)}=\begin{bmatrix} Z_{j-1}^{(i)}  \quad & Z_d \end{bmatrix}$,\\
Compute $\mathcal{W}_{i+1}^{(i)} = \mathcal{W}_{i-1}^{(i)} -4\text{Re}(\mu_j^{(i)}) E_1^T \delta_j^{(i)}$,\\
Compute $K_{j+1}^{(i)}=K_{j-1}^{(i)}+(B_1 - J_2{J_4}^{-1}B_2)^T Z_d {Z_d}^T E_1$,\\
$j=j+1$.
}
}
Update $Z^{(i)}=Z_j^{(i)}$ and $K^{(i)}=K_j^{(i)}$.	
}
\end{algorithm} 

\subsection{Computation of the optimally stabilized system and the optimal control}

The optimal feedback matrix $K^o=\mathcal{B}^T X\mathcal{E}$ can be achieved by the feasible solution $X$ of the Riccati equation (\ref{eqn:D-GCARE}). Then, applying $\mathcal{A}_s = \mathcal{A} - \mathcal{B} K^o$, optimally stabilized LTI continuous-time system can be written as (\ref{eqn:st-state-space}). To preserve the structure of the system, it needs to back to the oginial structure (\ref{eqn:DS_matrix_vector}), and for this the submatrices $J_1, J_3$ are replaced by $J_1 - B_1 K^o, J_3 - B_2 K^o$, respectively. 

Finally, the desired optimal control $u^o = -K^o x(t)$ of the targeted power system models can be computed.  

\section{Numerical Results}

The stability of the target models is investigated and the unstable models are stabilized through the Riccati-based feedback stabilization process. The proposed methods are employed to find the solution of the Riccati equation arising from the BIPS models and corresponding feedback matrices are generated for system stabilization. Also, initial Bernoulli feedback stabilization is implemented for the convenient rate of convergence.  
All the results have been achieved using the MATLAB 8.5.0 (R2015a) on a Windows machine having Intel-Xeon Silver 4114 CPU \@$2.20$ GHz clock speed, $2$ cores each and $64$ GB of total RAM.

\subsection{Brazilian inter-connected power system models}\label{bips}

Power system models are an essential part of engineering fields that consists of simulations based on power generations and grid networks. The computation required to analyze electrical power systems employing mathematical models utilizing real-time data. There are several applications of the power system model, i.e., electric power generation, utility transmission and distribution, railway power systems, and industrial power generation \cite{martins2007computation}.

The Brazilian Inter-connected Power Systems (BIPS) is one of the most convenient examples of the power system models with various test systems \cite{freitas1999computationally}. The following Table~\ref{tab:BIPS_systems} provides the details about the models. The detailed structure of those will be found at \textit{https://sites.google.com/site/rommes/software}, where all of them are index-1 descriptor system. The models $mod-606$, $mod-1998$, $mod-2476$ and $mod-3078$ have the unstable eigenvalues, whereas the models $mod-1142$, $mod-1450$ and $mod-1963$ have stable eigenvalues \cite{leandro2015identification}. Here the names of the models are considered according to their number of states.
\begin{table}[H]
\centering
\caption{Structure of the Models derived from BIPS test systems}
\label{tab:BIPS_systems}
\begin{tabular}{|c|c|c|c|c|c|}
\hline
\begin{tabular}[c]{@{}c@{}}Test \\ systems\end{tabular} & Dimensions & States & \begin{tabular}[c]{@{}c@{}}Algebraic\\ variables\end{tabular} & Inputs & Outputs \\ \hline
\multirow{4}{*}{BIPS98} & 7135 & 606 & 6529 & 4 & 4 \\ \cline{2-6} 
		& 9735 & 1142 & 8593 & 4 & 4 \\ \cline{2-6} 
		& 11265 & 1450 & 11582 & 4 & 4 \\ \cline{2-6} 
		& 13545 & 1963 & 13068 & 4 & 4 \\ \hline
\multirow{3}{*}{BIPS07} & 15066 & 1998 & 13068 & 4 & 4 \\ \cline{2-6} 
		& 16861 & 2476 & 14385 & 4 & 4 \\ \cline{2-6} 
		& 21128 & 3078 & 18050 & 4 & 4 \\ \hline
\end{tabular}
\end{table}

\subsection{Comparison of the results found by RKSM and KN-LRCF-ADI}

The CAREs arising from the models $mod-606$, $mod-1998$ and $mod-2476$ are efficiently solved and stabilized the corresponding models by both RKSM and KN-LRCF-ADI techniques. As model $mod-3078$ is semi-stable, the computation of CARE derived from this model is not possible by LRCF-ADI techniques but by the RKSM approach model $mod-3078$ successfully stabilized and the numerical result for model $mod-3078$ is investigated for RKSM only. 
\begin{table}[H]
\centering
\caption{Results of RKSM applied BIPS models}
\label{tab:rksm}
\begin{tabular}{|c|c|c|c|c|}\hline
Model & \begin{tabular}[c]{@{}c@{}}No of \\ iterations\end{tabular} & Tolerance & \begin{tabular}[c]{@{}c@{}}Numerical\\ rank\end{tabular} & \begin{tabular}[c]{@{}c@{}}CPU time\\ (second)\end{tabular} \\ \hline
		$mod-606$ & 100 & $10^{-10}$ & 195 & $1.81 \times 10^2$ \\ \hline
		$mod-1998$ & 200 & $10^{-10}$ & 266 & $1.41 \times 10^3$ \\ \hline
		$mod-2476$ & 248 & $10^{-10}$ & 265 & $3.06 \times 10^3$ \\ \hline	$mod-3078$ & 257 & $10^{-5}$ & 295 & $3.01 \times 10^3$ \\ \hline
\end{tabular}	
\end{table}

Table~\ref{tab:rksm} depicts the numerical results of the stabilization process via RKSM for the unstable BIPS models and various properties of the stabilized systems are illustrated, whereas Table~\ref{tab:kn-adi} displays the several modes of ADI techniques in KN-LRCF-ADI method for stabilizing the unstable BIPS models including characteristics of the stabilized models. In both of the tables, Table~\ref{tab:rksm} and Table~\ref{tab:kn-adi} we have analyzed the same features of the stabilized BIPS models, we can easily compare the efficiency and robustness of the proposed methods. 

From the above tables it can be said that the proposed RKSM approach has quick convergence ability and occupies very small solution space to provide the efficient solution of the CAREs. In contrast LRCF-ADI based Kleinman-Newton has several approaches for finding the solution of CAREs, whereas most of the approaches required higher computation time. Also, there are deviations in the numerical ranks of the factored solution of CAREs in the Kleinman-Newton approaches and in all of the cases RKSM provides significantly better result.
\begin{table}[H]
\centering
\caption{Results of KN-LRCF-ADI applied BIPS models}
\label{tab:kn-adi}
\begin{tabular}{|c|c|c|c|c|}\hline
Model & \begin{tabular}[c]{@{}c@{}}Total\\ iterations\end{tabular} & Tolerance & \begin{tabular}[c]{@{}c@{}}Numerical\\ rank\end{tabular} & \begin{tabular}[c]{@{}c@{}}CPU time\\ (second)\end{tabular} \\ \hline
\multirow{4}{*}{$mod-606$} & 311 & $10^{-5}$ & 481 & $2.74 \times 10^2$ \\ \cline{2-5} 
		& 544 & $10^{-10}$ & 953 & $7.21 \times 10^2$ \\ \cline{2-5} 
		& 508 & $10^{-5}$ & 473 & $4.10 \times 10^2$ \\ \cline{2-5} 
		& 853 & $10^{-10}$ & 969 & $7.92 \times 10^2$ \\ \hline
\multirow{4}{*}{$mod-1998$} & 277 & $10^{-5}$ & 663 & $2.30 \times 10^3$ \\ \cline{2-5} 
		& 485 & $10^{-10}$ & 1201 & $5.83 \times 10^3$ \\ \cline{2-5} 
		& 514 & $10^{-5}$ & 497 & $3.27 \times 10^3$ \\ \cline{2-5} 
		& 1003 & $10^{-10}$ & 1417 & $1.17 \times 10^4$ \\ \hline
\multirow{4}{*}{$mod-2476$} & 254 & $10^{-5}$ & 473 & $3.13 \times 10^3$ \\ \cline{2-5}
		& 363 & $10^{-10}$ & 937 & $5.75 \times 10^3$ \\ \cline{2-5} 
		& 366 & $10^{-5}$ & 441 & $3.02 \times 10^3$ \\ \cline{2-5} 
		& 849 & $10^{-10}$ & 849 & $9.50 \times 10^4$ \\ \hline
\end{tabular}
\end{table}

\subsection{Stabilization of eigenvalues}

Fig.\ref{fig:eig_606}, Fig.\ref{fig:eig_1998}, and Fig.\ref{fig:eig_2476} illustrates the eigenvalue stabilization of the models $mod-606$, $mod-1998$, $mod-2476$ by both the RKSM and LRCF-ADI based Kleinamn-Newton techniques.

\begin{figure}[H]
	\centering 
	\begin{subfigure}{.95\linewidth}
		\includegraphics[width=\linewidth]{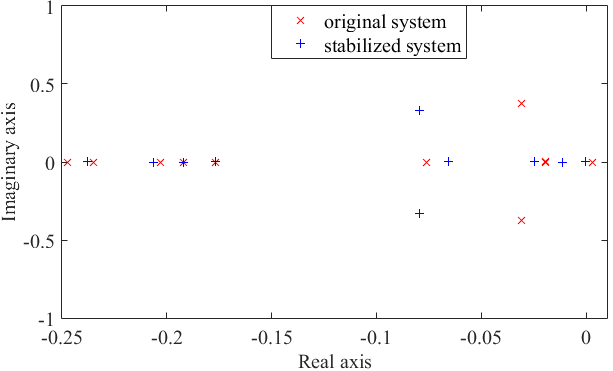}
		\vspace{0.01in}
		\caption{Stabilized by RKSM} 
		\vspace{0.1in}
		\label{fig:eig606_1}
	\end{subfigure}
	\vspace{0.1in}
	\begin{subfigure}{.95\linewidth}
		\includegraphics[width=\linewidth]{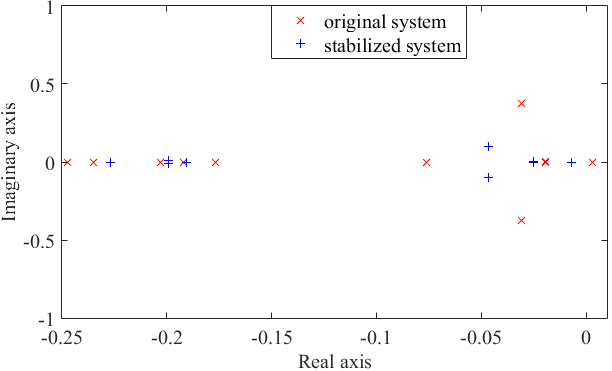}
		\vspace{0.01in}
		\caption{Stabilized by KN-LRCF-ADI} 
		\vspace{0.1in}
		\label{fig:eig606_2}
	\end{subfigure}
	\vspace{0.1in}
	\caption{Comparisons of the eigenvalues for the model $mod-606$}
	\label{fig:eig_606}
\end{figure} 
\begin{figure}[H]
	\centering 
	\begin{subfigure}{.95\linewidth}
		\includegraphics[width=\linewidth]{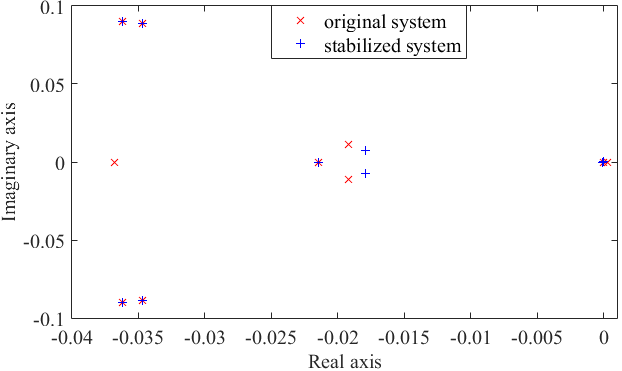}
		\vspace{0.01in}
		\caption{Stabilized by RKSM} 
		\vspace{0.1in}
		\label{fig:eig1998_1}
	\end{subfigure}
	\vspace{0.1in}
	\begin{subfigure}{.95\linewidth}
		\includegraphics[width=\linewidth]{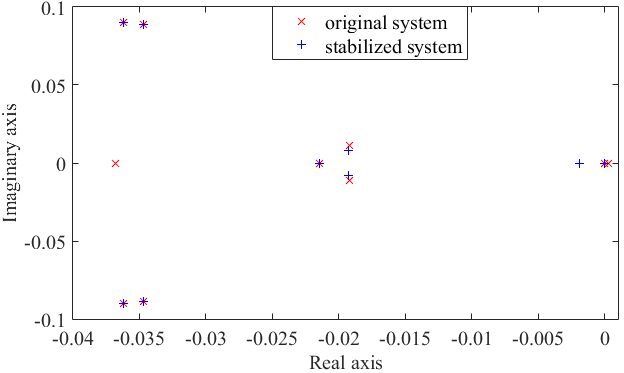}
		\vspace{0.01in}
		\caption{Stabilized by KN-LRCF-ADI} 
		\vspace{0.1in}
		\label{fig:eig1998_2}
	\end{subfigure}
	\vspace{0.1in}
	\caption{Comparisons of the eigenvalues for the model $mod-1998$}
	\label{fig:eig_1998}
\end{figure} 

From the sub-figures in the above mentioned figures, it has been observed that the eigenvalues of the models $mod-606$ and $mod-2476$ are stabilized very well but its marginal for the model $mod-1998$. Thus, it can be concluded that both the RKSM and LRCF-ADI based Kleinamn-Newton techniques have adequate efficiency to stabilize the unstable descriptor systems by closed-loop structure via Riccati based feedback stabilization. 

Fig.\ref{fig:eig_3078} displays the applicability of the RKSM for the semi-stable descriptor system, whereas LRCF-ADI based methods are ineffective in this case. Since the eigenvalues of the semi-stable model $mod-3078$ are very close to the imaginary axis, it is not possible to realize in the normal view. So, for the simulation tool capacity and visual convenience, we have considered the a very magnified view of the eigenspaces. 

\begin{figure}[H]
	\centering 
	\begin{subfigure}{.95\linewidth}
		\includegraphics[width=\linewidth]{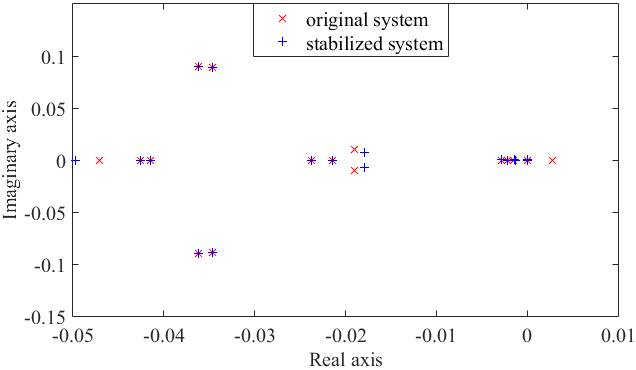}
		\vspace{0.01in}
		\caption{Stabilized by RKSM} 
		\vspace{0.1in}
		\label{fig:eig2476_1}
	\end{subfigure}
	\vspace{0.1in}
	\begin{subfigure}{.95\linewidth}
		\includegraphics[width=\linewidth]{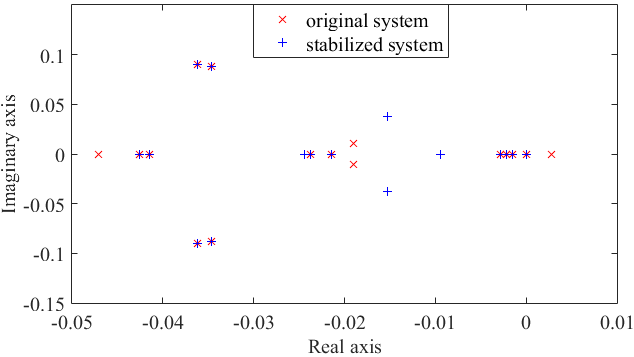}
		\vspace{0.01in}
		\caption{Stabilized by KN-LRCF-ADI} 
		\vspace{0.1in}
		\label{fig:eig2476_2}
	\end{subfigure}
	\vspace{0.1in}
	\caption{Comparisons of the eigenvalues for the model $mod-2476$}
	\label{fig:eig_2476}
\end{figure} 
\begin{figure}[H]
	\centering 
	\includegraphics[width=\linewidth]{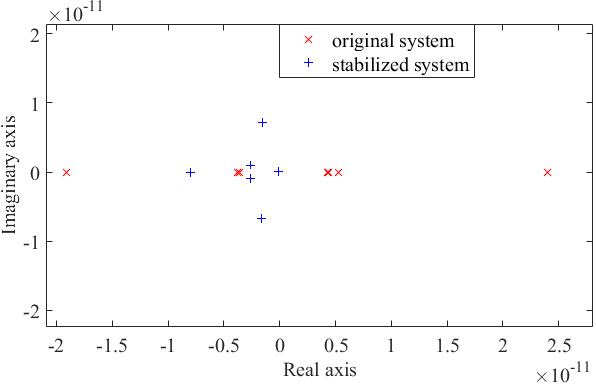}
	\vspace{0.01in}
	\caption{Comparisons of the eigenvalues (stabilized by RKSM) for the model $mod-3078$}
	\vspace{0.1in}
	\label{fig:eig_3078}
\end{figure} 

\subsection{Stabilization of step-responses}

The investigation of the figures from Fig.\ref{fig:step-res_606_1}-Fig.\ref{fig:step-res_3078_2} consist the step-responses for some dominant input/output relations to compare the RKSM and the LRCF-ADI based Kleinman-Newton approaches via the system stabilization. Since the power system models are of $4\times4$ input-output relations, there are $16$ step-responses for each models. For the effective comparison we have investigated only some graphically significant step-responses. 

\begin{figure}[H]
	\centering 
	\includegraphics[width=\linewidth]{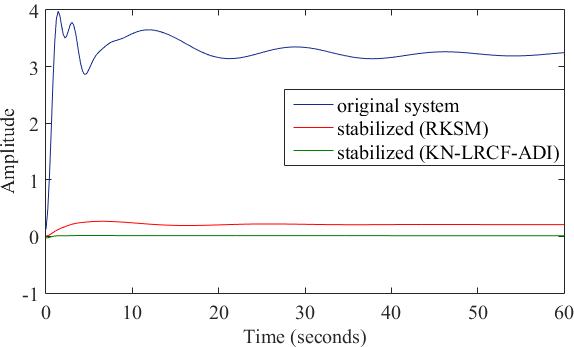}
	\vspace{0.01in}
	\caption{Comparisons of step-responses for the model $mod-606$ for the first-input/third-output}
	\label{fig:step-res_606_1}
\end{figure}
\begin{figure}[H]
	\centering 
	\includegraphics[width=\linewidth]{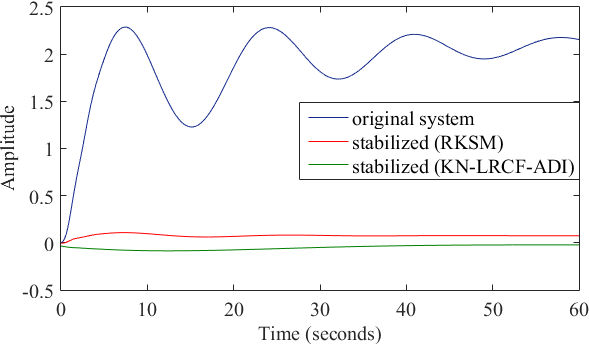}
	\vspace{0.01in}
	\caption{Comparisons of step-responses for the model $mod-606$ for the second-input/first-output}
	\label{fig:step-res_606_2}
\end{figure}
\begin{figure}[H]
	\centering
	\includegraphics[width=\linewidth]{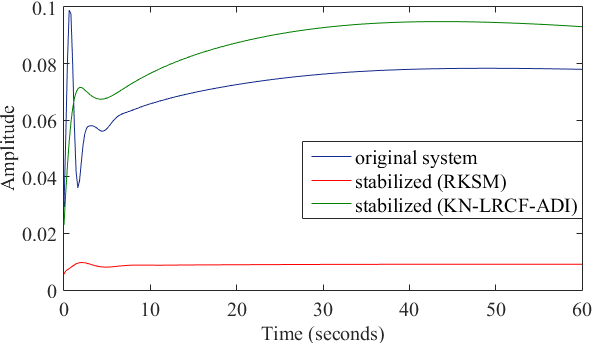}
	\vspace{0.01in} 	
	\caption{Comparisons of step-responses for the model $mod-1998$ for the third-input/second-output}
	\label{fig:step-res_1998_1}
\end{figure}
\begin{figure}[H]
	\centering
	\includegraphics[width=\linewidth]{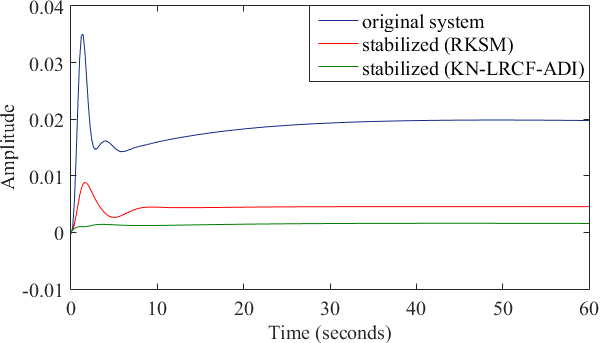}
	\vspace{0.01in} 	
	\caption{Comparisons of step-responses for the model $mod-1998$ for the fourth-input/third-output}
	\label{fig:step-res_1998_2}
\end{figure}
\begin{figure}[H]
	\centering 
	\includegraphics[width=\linewidth]{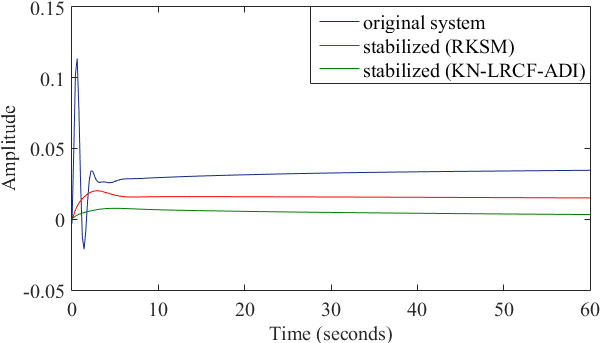}
	\vspace{0.01in}
	\caption{Comparisons of step-responses for the model $mod-2476$ for the first-input/third-output}
	\label{fig:step-res_2476_1}
\end{figure}
\begin{figure}[H]
	\centering 
	\includegraphics[width=\linewidth]{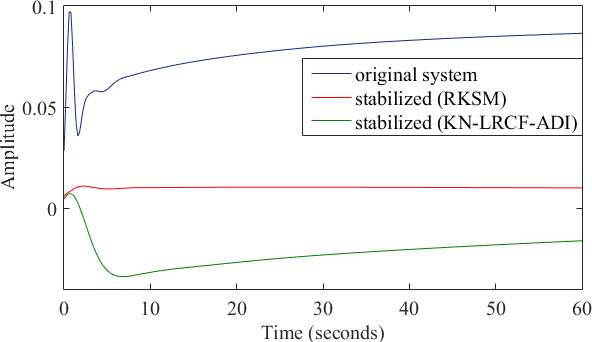}
	\vspace{0.01in}
	\caption{Comparisons of step-responses for the model $mod-2476$ for the third-input/second-output}
	\label{fig:step-res_2476_2}
\end{figure}

From Fig.\ref{fig:step-res_606_1}-Fig.\ref{fig:step-res_2476_2} it is evident that both RKSM and LRCF-ADI integrated Kleinman-Newton techniques are applicable for the Riccati based feedback stabilization of unstable power system models. The graphical comparisons indicates by RKSM is suitably robust. On the other hand, though sometimes the Kleinman-Newton approach provides very good accuracy but it has some scattered behaviors. 

Moreover, Fig.\ref{fig:step-res_3078_1} and Fig.\ref{fig:step-res_3078_2} shows the applicability of the RKSM technique for the Riccati based feedback stabilization for the semi-stable index-1 descriptor systems.  

\begin{figure}[H]
	\centering 
	\includegraphics[width=\linewidth]{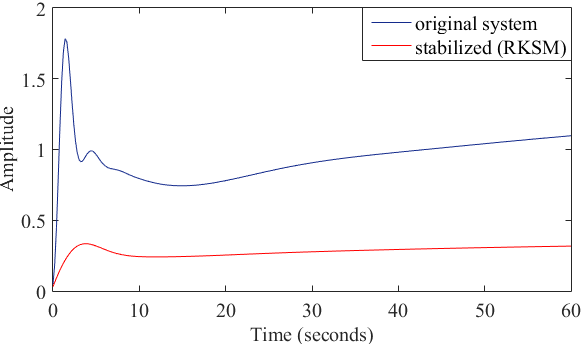}
	\vspace{0.1in}
	\caption{Comparisons of step-responses for the model $mod-3078$ for the second-input/first-output}
	\label{fig:step-res_3078_1}
\end{figure}
\begin{figure}[H]
	\centering 
	\includegraphics[width=\linewidth]{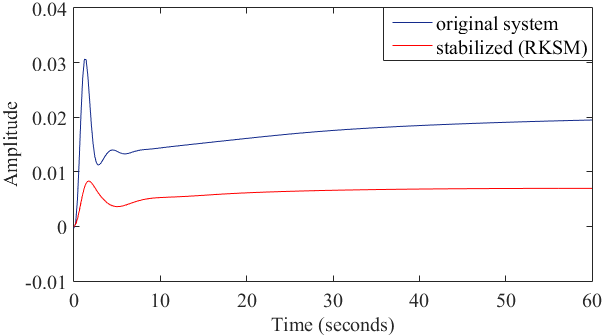}
	\vspace{0.1in}
	\caption{Comparisons of step-responses for the model $mod-3078$ for the fourth-input/third-output}
	\label{fig:step-res_3078_2}
\end{figure}

\section{Conclusion}

From the tabular and graphical comparisons of the results of numerical computations, we have observed that by both RKSM and KN-LRCF-ADI techniques, the CAREs arising from the unstable index-1 descriptor systems are efficiently solved and the corresponding models are properly stabilized. The semi-stable index-1 descriptor system is successfully stabilized through Riccati-based feedback stabilization by RKSM, whereas KN-LRCF-ADI is still not suitable for it. There are deviations of the numerical ranks of the factored solutions of CAREs in the Kleinman-Newton approaches, while RKSM provides significantly better results for all the cases. RKSM approach has quick convergence ability and occupies very small solution spaces to provide efficient solutions to the CAREs. In contrast, LRCF-ADI based Kleinman-Newton has several approaches for finding the solutions of CAREs, where almost all of the approaches required higher computation time. Riccati-based feedback stabilization for the index-1 descriptor systems by the RKSM approach is very effective and robust. Contrariwise, LRCF-ADI based Kleinman-Newton method is slightly scattered in case of the stabilization of step-responses. Thus, it can be concluded that the RKSM is suitably applicable to the unstable index-1 descriptor systems for Riccati-based feedback stabilization and this method is more preferable to the Kleinman-Newton method in the sense of computation time and memory allocation.

In this work, the MATLAB library command \texttt{care} is used in RKSM to find the solution of the CAREs governed from the reduced-order models and used the direct backward inversion technique for solving shifted sparse linear systems. In the future, we will try to find the self-sufficient RKSM algorithm for solving CAREs and apply the "Restarted Hessenberg Method" for solving shifted sparse linear systems.  Since the LRCF-ADI techniques are inefficient for the semi-stable systems, we will work for it as well.

\section{Acknowledgment}

An earlier version of it has been presented as a pre-print in "Riccati-based feedback stabilization for unstable power system models". The earlier version of this work can be found at the link https://arxiv.org/abs/2006.14210.

This work is under the project \textit{"Computation of Optimal Control for Differential-Algebraic Equations (DAE) with Engineering Applications"}. This project is funded by United International University, Dhaka, Bangladesh. It starts from October 01, 2019, and the reference is IAR/01/19/SE/18. 

\bibliographystyle{IEEEtran}
\bibliography{Ref_MUn}

\end{document}